\documentclass[12pt,a4paper]{amsart}
\usepackage{amsmath,amsthm,amssymb,latexsym}
\usepackage{graphicx}

\newcommand{\bd}{{\mathbb{D}}}
\newcommand{\bn}{{\mathbb{N}}}
\newcommand{\br}{{\mathbb{R}}}

\newcommand{\bc}{{\mathbb{C}}}
\newcommand{\bbc}{{\bar{\mathbb{C}}}}

\newcommand{\bt}{{\mathbb{T}}}

\newcommand{\fe}{{\mathfrak{e}}}

\newcommand{\ca}{{\mathcal{A}}}
\newcommand{\cf}{{\mathcal{F}}}
\newcommand{\css}{{\mathcal{S}}}

\renewcommand{\a}{\alpha}
\renewcommand{\b}{\beta}
\renewcommand{\l}{\lambda}
\newcommand{\s}{\sigma}

\renewcommand{\d}{\delta}
\newcommand{\dd}{\Delta}

\newcommand{\oo}{\Omega}
\newcommand{\g}{\gamma}
\renewcommand{\gg}{\Gamma}
\newcommand{\eps}{\varepsilon}
\newcommand{\z}{\zeta}

\newcommand{\nt}{\noindent}

\newcommand{\bsl}{\backslash}

\newcommand{\pt}{\partial}
\newcommand{\ti}{\tilde}

\newcommand{\lp}{\left(}
\newcommand{\rp}{\right)}
\newcommand{\lt}{\left}
\newcommand{\rt}{\right}

\DeclareMathOperator{\di}{\rm d}

\allowdisplaybreaks
\numberwithin{equation}{section}

\newtheorem{theorem}{Theorem}[section]

\newtheorem{lemma}[theorem]{Lemma}
\newtheorem{corollary}[theorem]{Corollary}

\theoremstyle{definition}

\newtheorem{remark}[theorem]{Remark}

\begin{document}

\title[A Blaschke-type condition and finite-band operators]
{A Blaschke-type condition for analytic functions on finitely connected domains.
Applications to complex perturbations of  a finite-band selfadjoint operator}
\author[L. Golinskii, S. Kupin]{L. Golinskii and S. Kupin}

\address{Mathematics Division, Institute for Low Temperature Physics and
Engineering, 47 Lenin ave., Kharkov 61103, Ukraine}
\email{leonid.golinskii@gmail.com}

\address{IMB, Universit\'e Bordeaux 1, 351 cours de la Lib\'eration, 33405 Talence Cedex, France}
\email{skupin@math.u-bordeaux1.fr}

\date{June, 3, 2011}

\thanks{The second author is partially supported by  ANR AHPI, ANR FRAB grants}

\begin{abstract}
This is a sequel of the article by Borichev--Golinskii--Kupin
\cite{bgk} where the authors obtain Blaschke-type conditions for
special classes of analytic functions in the unit disk which satisfy
certain growth hypotheses. These results were applied  to get
Lieb--Thirring inequalities for complex compact perturbations of a
selfadjoint operator with a simply connected resolvent set.

The first result of the present paper is an appropriate local version of the Blaschke-type condition from \cite{bgk}. We apply it to obtain a similar condition for an analytic function in a finitely
connected domain of a special type. Such condition is by and large the same as a Lieb--Thirring type
inequality for complex compact perturbations of a selfadjoint
operator with a finite-band spectrum. A particular case of this
result is the Lieb--Thirring inequality for a selfadjoint
perturbation of the Schatten class of a periodic (or a finite-band) Jacobi matrix. The latter result seems to be new in such generality even in this framework.
\end{abstract}

\maketitle

\vspace{-0.5cm}
\section*{Introduction}
\label{s0}

Let $e=\{\a_j, \b_j\}_{j=1,\dots, n+1}\subset \br$ be a set of
distinct points. We suppose that
\begin{equation}\label{e0}
-\infty <\a_1<\b_1<\a_2<\b_2<\dots <\a_{n+1}<\b_{n+1}< +\infty.
\end{equation}
Let also
\begin{equation}\label{e01}
\fe=\bigcup^{n+1}_{j=1} \fe_j, \quad \fe_j=[\a_j,\b_j],
\end{equation}
and $\oo:=\bbc\bsl \fe$. For a function $f$
analytic in $\oo$, $f\in \ca(\oo)$, $Z_f$ stands for the set of
the zeros counting the multiplicities. By $\di(\l,M)$ we
denote the distance between a point $\l$ and a set $M$.

Our main functional theoretic result looks as follows.

\begin{theorem}\label{t2}
Let $f\in \ca(\oo), \ |f(\infty)|=1$, and, for $p, q\ge 0$
\begin{equation}\label{e3}
\log |f(\l)|\le \frac{K_1}{\di^p(\l,\fe) \di^q(\l,e)}.
\end{equation}
Then for any $0<\eps<1$
\begin{equation}\label{e4}
\sum_{\l\in Z_f}
\di^{p+1+\eps}(\l,\fe)\, \di^{a(p,q,\eps)}(\l,e)\, (1+|\l|)^{b(p,q,\eps)}\le
C\cdot K_1,
\end{equation}
where
\begin{equation*}
\begin{split}
a(p,q,\eps) &=\frac{(p+2q-1+\eps)_+ -(p+1+\eps)}2\,, \\
b(p,q,\eps) &=(p+q-1+\eps)_+-\frac{(p+2q-1+\eps)_+ +p+1+\eps}2\,.
\end{split}
\end{equation*}
\end{theorem}
As usual, $x_+=\max\{x, 0\}$.
Here and in the sequel $C=C(\fe,p,q,\eps)$ stands for a generic positive constant which depends on indicated parameters.
Of course, inequality \eqref{e4} looks somewhat cumbersome, and it can be simplified in specific situations. Here are two examples.

\begin{corollary}\label{t02}
Let $f\in \ca(\oo), \ |f(\infty)|=1$, and, for $p, q\ge 0$,
$p+q\ge1$
\begin{equation*}
\log |f(\l)|\le \frac{K_1}{\di^p(\l,\fe) \di^q(\l,e)}.
\end{equation*}
Then for any $0<\eps<1$
\begin{equation}\label{e004}
\sum_{\l\in Z_f}
\frac{\di^{p+1+\eps}(\l,\fe)\, \di^{q-1}(\l,e)}{1+|\l|}\le
C\cdot K_1.
\end{equation}
\end{corollary}

The case $q=0$ is important for applications.

\begin{corollary}\label{t1}
Let $f\in \ca(\oo), \ |f(\infty)|=1$, and
\begin{equation}\label{e1}
\log |f(\l)|\le \frac{K_1}{\di^p(\l, \fe)}\,, \quad p\ge0.
\end{equation}
Then for any $0<\eps<1$
\begin{equation}\label{e2}
\sum_{\l\in Z_f}
\frac{\di^{p+1+\eps}(\l,\fe)}{\di(\l,e) (1+|\l|)}\le C\cdot K_1,
\end{equation}
as long as $p\ge1$, and
\begin{equation}\label{e21}
\sum_{\l\in Z_f}
\frac{\di^{p+1+\eps}(\l,\fe)}{(\di(\l,e)\, (1+|\l|))^{(p+1+\eps)/2}}\le C\cdot K_1
\end{equation}
for $p<1$.
\end{corollary}

\bigskip

All operators appearing in the present paper act on a separable Hilbert space $H$. Consider a (bounded) selfadjoint operator $A_0$ defined on $H$. We suppose it to be finite-band, i.e., for its spectrum
$$ \s(A_0)=\s_{ess}(A_0)=\fe, $$
where $\fe$ looks like in \eqref{e01}. A typical example here is a double
infinite periodic Jacobi matrix. Let $B\in \css_p$, the Schatten
class of operators, $p\ge 1$. We do not suppose $B$ to be
selfadjoint. By the Weyl theorem, see, e.g., \cite{kr1}, the
essential spectrum $\s_{ess}(A)$, $A=A_0+B$, coincides with
$\s_{ess}(A_0)$.

We want to have some information on the distribution of the
discrete spectrum $\s_d(A):=\s(A)\backslash\s_{ess}(A)$, which
consists of eigenvalues of finite algebraic multiplicity. It is
clear that the points from $\s_d(A)$ can only accumulate to $\fe$.
Here is the quantitative version of this intuition.

\begin{theorem}\label{t3} Let $A_0$ be as described above, $B\in \css_p$ and $A=A_0+B$. Then, for $0<\eps<1$ and $p\ge1$
\begin{equation}\label{e5}
\sum_{\l\in\s_d(A)}
\frac{\di^{p+1+\eps}(\l,\fe)}{\di(\l,e)\, (1+|\l|)}\le C\cdot
\|B\|_{\css_p},
\end{equation}
and for $0\le p<1$
\begin{equation}\label{e51}
\sum_{\l\in Z_f}
\frac{\di^{p+1+\eps}(\l,\fe)}{(\di(\l,e)\, (1+|\l|))^{(p+1+\eps)/2}}\le C\cdot \|B\|_{\css_p}.
\end{equation}
\end{theorem}

In such generality the above inequality seems to be new even for the case when $A$ is a selfadjoint
perturbation of a periodic selfadjoint Jacobi matrix $A_0$.

\begin{remark} The case $n=0$, i.e., $\s(A_0)=[\a,\b]$, is not
exceptional. The point is that for $e=\{\a,\b\}$
$$ C_1|(\l-\a)(\l-\b)|\le \di(\l,e)(1+|\l|) \le
C_2|(\l-\a)(\l-\b)|, \quad \l\in\bc, $$ 
with  absolute
constants  $C_{1,2}$, so we come to Theorem 2.3 from \cite{bgk}.
\end{remark}

For the Lieb--Thirring inequalities for nonselfadjoint compact
perturbations of the discrete Laplacian see also Golinskii--Kupin \cite{gk},
Hansmann--Katriel \cite{hk}. A few interesting results of the same flavor on Lieb--Thirring inequalities for selfadjoint Jacobi matrices and Schr\"odinger operators are in  Hun\-dert\-mark--Simon \cite{hu1}, Damanik--Killip--Simon \cite{da1} and Frank--Simon \cite{fra1}.

\smallskip
As usual, we write $\bd=\{z: |z|<1\}$ for the unit disk,
$\bt=\{z: |z|=1\}$ for the unit circle, and $B(w_0,r)=\{w: |w-w_0|< r\}$ for balls in the complex plane. Sometimes, we label the balls by the variable of the corresponding complex plane, i.e. $B_w(z_0,r)$ ($B_\l(z_0, r)$) stays for a ball in the $w$-plane (the $\l$-plane), respectively.

\section{Local version of Borichev--Golinskii--Kupin Theorem}

We begin with the result of Borichev--Golinskii--Kupin
\cite[Theorem 0.2]{bgk} and its version in \cite[Theorem
4]{hk}.
\begin{theorem}\label{bgkhk}
Let $I=\{\z_j\}_{j=1,\dots, k}$ be a finite subset of $\bt$, $f\in
\ca(\bd)$, $|f(0)|=1$, and for $p',q',s\ge 0$
$$
\log |f(z)|\le \frac {K|z|^s}{\di^{p'}(z,\bt)\, \di^{q'}(z, I)}\,,
\quad z\in \bd.
$$
Then for any $0<\eps<1$
$$
\sum_{z\in Z_f} \frac{\di^{p'+1+\eps}(z,\bt)}{|z|^{(s-1+\eps)_+}}\, \di^{(q'-1+\eps)_+}(z,I)\le C(I,p',q',\eps)\cdot K.
$$
\end{theorem}

Our goal here is to prove a local version of the above
result (cf. \cite[Theorem 7]{fg2}).

Let $G\subset \bar\bd$ be an open circular polygon, $0\in G$, with
vertices $\{u_i\}\in\bt$, and sides (arcs) $\tau_i=[u_i,u_{i+1}]$,
$i=1,2,\ldots,2N$, $u_{2N+1}=u_1$ (see Figure \ref{f1}). The arcs
$\tau_{2j}$ lie on $\bt$, and $\tau_{2j-1}$, which we call the
inner sides of $G$, lie on some orthocircles, that is, circles
orthogonal to $\bt$. Put
$$
\dd_1=\pt G \cap \bd=\lt\{\bigcup^N_{j=1} (u_{2j-1}, u_{2j})
\rt\}, \quad \dd_2=\pt G\cap\bt=\lt\{\bigcup^N_{j=1} [u_{2j},
u_{2j+1}] \rt\},
$$
so
$$
\pt G=\dd_1\cup\dd_2.
$$

Let  $E=\{\z_j\}_{j=1,\dots, k}\subset \dd_2$ be a selected finite
subset of the unit circle. We take $\ti G\subset G$ to be a
properly ``shrunk'' circular polygon, in such a way that $E\subset
\ti\dd_2$, see again Figure \ref{f1}. The notation for $\ti G$ is
the same as for $G$ up to ``waves" referring to the first set. So,
for instance, the vertices of $\ti G$ are $\ti u_i$,
$$
\pt\ti G=\ti\dd_1\cup\ti\dd_2, \quad \ti\dd_1= \pt \ti G \cap \bd,
\quad \ti\dd_2=\pt\ti G \cap \bt.
$$
It is important that $\min_j \di(\tau_{2j-1},\ti
\tau_{2j-1})=\di'>0$, $j=1,2,\ldots,N$.

Consider a conformal map $w, \ w: \bd\to G$, normalized by
$w(0)=0$, $w'(0)>0$. Sometimes, to indicate explicitly the
variables, we will write $w: \bd_z\to G_w$.

Put $\ti D=w^{-1}(\ti G) \subset \bd_z$ and introduce
\begin{itemize}
\item preimages of vertices $v_j=w^{-1}(u_j), \ti v_j=w^{-1}(\ti
u_j), \quad j=1,\dots, 2N$, \item preimages of sides
$\ti\tau_j=w^{-1}(\tau_j)\subset \bt_z, \quad j=1,\dots, 2N$.
\item preimages of selected points $I=\{\xi_j=w^{-1}(\z_j)\}$,
$j=1,\dots, k$.
\end{itemize}
Clearly, $I$ is contained in the closure of $\ti D$.

\begin{figure}[t]
\includegraphics[width=14cm]{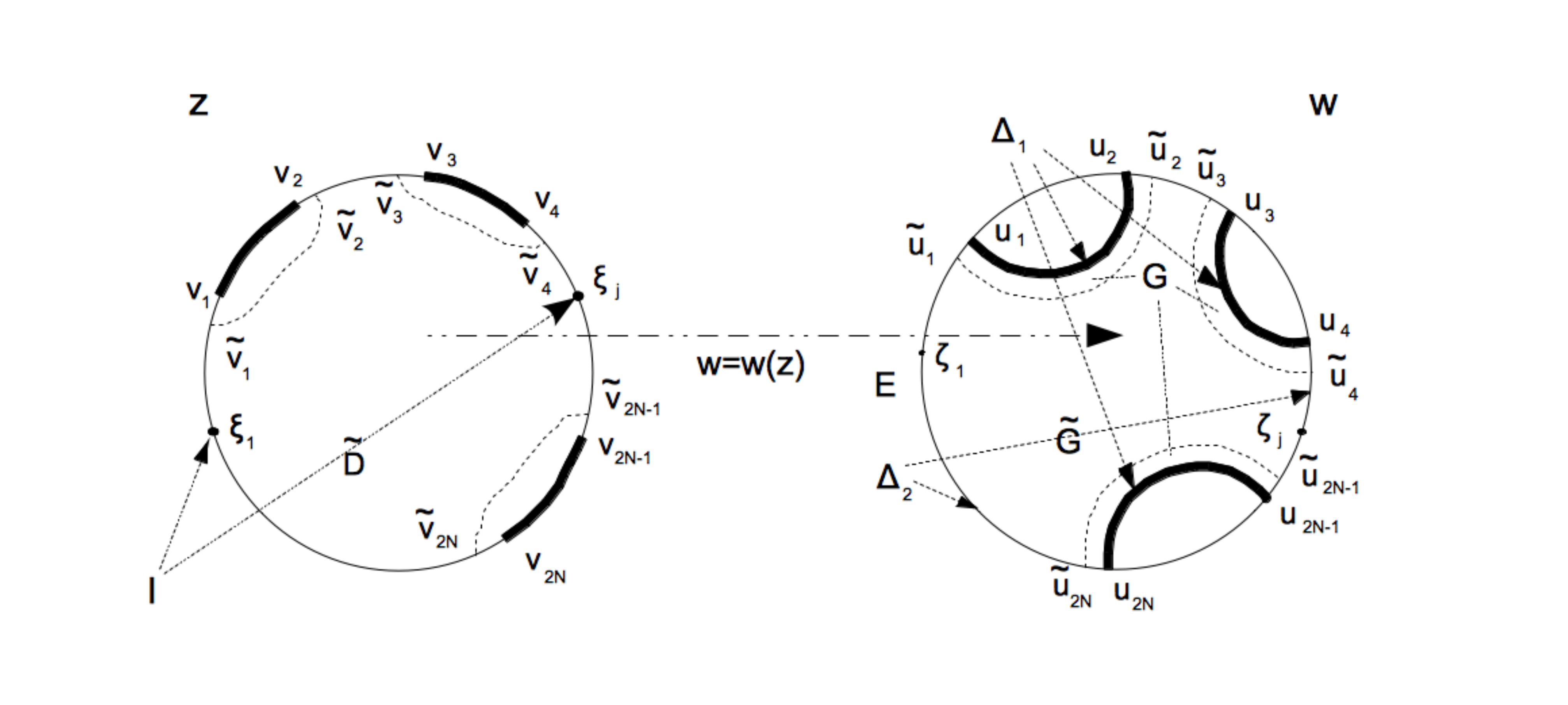}
\caption{The domains $G, \ti G$ and the map $w$.}\label{f1}
\end{figure}

For short, we write $w=w(z)$. Here is a couple of elementary properties of $w$:
\begin{itemize}
\item $\di(z,\bt_z)=1-|z|\le 1-|w|=\di(w,\bt_w)$ by the Schwarz
lemma. \item By \cite[Corollary 1.4]{pom}, $\di(w,\pt G)\asymp
|w'(z)|\, (1-|z|)=|w'(z)|\, \di(z,\bt)$. Since $z\in \ti D$ if and
only if $w\in\ti G$, and $|w'(z)|\asymp 1$ for $z\in\ti D$, then
\begin{equation}\label{0001}
\di(w,\pt G)\asymp \di(z,\bt), \quad z\in\ti D. \end{equation}
\end{itemize}
Here and in what follows the equivalence relation $A\asymp B$
means that $c_1\le A/B \le c_2$ for generic positive constants
$c_i$ which depend only on $G$ and $E$. Similarly,
\begin{equation}\label{001}
\di(w,E)\asymp \di(z,I), \quad z\in\bd.
\end{equation}
Indeed, for $z\in\ti D$, $w\in\ti G$, we have $|w'(z)|,
|z'(w)|\asymp 1$. For $z\in \bd\backslash{\ti D}$ both sides in
\eqref{001} are equivalent to $1$.

\medskip

Let now $f\in \ca(G)$, $|f(0)|=1$, and assume that for some
$p',q',s\ge 0$
\begin{equation}\label{e6}
\log |f(w)|\le \frac {K|w|^s}{\di^{p'}(w,\bt)\, \di^{q'}(w, E)}\,,
\quad w\in G.
\end{equation}
Consider a function $F(z)=f(w(z))\in \ca(\bd_z)$. By using the first
property of $w$, equivalence $|w|\asymp |z|$, and \eqref{001}, we
obtain
$$
\log |F(z)|\le \frac{K|z|^s}{\di^{p'}(z,\bt)\, \di^{q'}(z, I)}\,,
\quad z\in\bd.
$$
Theorem \ref{bgkhk} now implies
$$
\sum_{z\in Z_F} \frac{\di^{p'+1+\eps}(z,\bt)}{|z|^{(s-1+\eps)_+}}\, \di^{(q'-1+\eps)_+}(z,I)\le C\cdot K
$$
for any $0<\eps<1$, and, by far,
$$
\sum_{z\in \ti D\cap Z_F} \frac{\di^{p'+1+\eps}(z,\bt)}{|z|^{(s-1+\eps)_+}}\, \di^{(q'-1+\eps)_+}(z,I)\le C\cdot K.
$$
Of course, $Z_f=w(Z_F)$, so by \eqref{0001} and \eqref{001}
$$
\sum_{w\in \ti G\cap Z_F}\frac{\di^{p'+1+\eps}(w, \pt G)}{|w|^{(s-1+\eps)_+}}\, \di^{(q'-1+\eps)_+}(w,E)\le
C\cdot K.
$$

Let us show that $\di(w,\pt G)\ge C\di(w,\bt_w)$, as long as
$w\in\ti G$. Indeed, if $\di(w,\pt G)=\di(w,\dd_2)$, then
$\di(w,\pt G)\ge \di(w, \bt_w)$. Otherwise, $\di(w,\pt G)=\di(w,
\dd_1)$, so $\di(w, \pt G)\ge \di'$ and
$$
\di(w,\bt_w)=1-|w|\le 1\le \frac{\di(w, \pt G)}{\di'}\,,
$$
as claimed. Hence
\begin{equation}\label{e7}
\sum_{w\in \ti G\cap Z_F} \frac{\di^{p'+1+\eps}(w,\bt_w)}{|w|^{(s-1+\eps)_+}}\, \di^{(q'-1+\eps)_+}(w,E)\le C\cdot K.
\end{equation}

That is, we have proven

\begin{theorem}\label{t4}
Let $f\in \ca(G), |f(0)|=1$, and for $p',q',s\ge 0$
$$
\log |f(w)|\le \frac {K|w|^s}{\di^{p'}(w,\bt)\, \di^{q'}(w, E)}\,,
\quad w\in G.
$$
Then, \eqref{e7} holds  for any $0<\eps<1$.
\end{theorem}

It goes without saying that the similar counterpart of Theorem 0.3
from \cite{bgk} is also valid in the present setting.

\section{Uniformization, Fuchsian groups, and all that}\label{s2}

In this section we are aimed at proving Theorem \ref{t2} with the
help of Theorem \ref{t4}.

We start reminding  the celebrated uniformization theorem of
Klein--Koebe--Poincar\'e \cite[Ch. III]{as}, which is one of the
key ingredients of the proof. The result is valid for arbitrary
Riemann surfaces, but we will formulate it for the so called
planar domains, since this is enough for our purposes. Recall that
a discrete group of  M\"obius transformations $\gg$ (of $\bd$ on
itself) is called a Fuchsian group. The discreteness means that
any orbit $\{\g(z)\}_{\g\in \gg}$ is a discrete set in the
relative topology of $\bd$.

Let $\oo\subset \bar\bc$  be a domain with the boundary containing
more than two points, and $\l_0\in\oo$. The uniformization theorem
says that there exists a covering map $\l: \bd\to \oo$, which is
unique provided the normalization conditions $\l(0)=\l_0$,
$\l'(0)>0$ are set. Moreover, the map is automorphic with respect
to a certain Fuchsian group $\gg$, i.e., $\l\circ\g=\l$ for any
$\g\in\gg$. Symbolically, we write
$$
\oo \simeq  \bd/\gg,
$$
where  two points $z,w\in\bd$ are equivalent with respect to $\gg$
if and only if there is a $\g\in\gg$ such that $w=\g(z)$. For
further terminology on the subject, we refer to \cite[Ch.
III]{as}, \cite{fo}; see also Simon \cite{si3} for a recent
presentation.

We will focus upon the special case
$\oo=\bar\bc\bsl\fe$, described in \eqref{e0}. The standard
normalization now is
\begin{equation}\label{122}
\l(0)=\infty, \qquad \lim_{w\to0}w\l(w)=\kappa(\fe)>0.
\end{equation}
The properties of the Fuchsian group $\gg$ in this  situation are
well-studied, see \cite[Chapter 9.6]{si3}. In particular, $\gg$ is
a free nonabelian group with $n$ generators $\{\g_j\}_{j=1}^n$.
\if{Every generator $\g_j$ is hyperbolic and given by the
reflection with respect to an orthocircle $c_j$ followed by the
conjugation. The circles $\ti c_j$ are disjoint, see Figure
\ref{f2}, and they define completely the group $\gg$.}\fi The
fundamental domain $\cf$ (more precisely, its interior
$\cf^{int}$) is a circular polygon in $\bd$, its topological
boundary in $\bd$ consists of $n$ orthocircles in $\bc_+$ and
their complex conjugates, and there is a finite distance in
$\overline\bd$ between the different orthocircles, see Figure 2.
We label the vertices of $\cf$ by $E=\l^{-1}(e)=\{w_j\}$.

The following relations for the covering map are crucial in the
sequel.
\begin{lemma}\label{lem1}
Let $w\in\overline\cf$, closure in $\overline\bd$, and $\l=\l(w)$.
Then
\begin{equation}\label{101}
\di(\l,e)\asymp \frac{\di^2(w,E)}{|w|} \end{equation} and
\begin{equation}\label{102}
\di(\l,\fe)\asymp \frac{\di(w,\bt_w)\di(w,E)}{|w|}\,.
\end{equation}
\end{lemma}

\begin{proof}
In the case $w\in B(0,r)$ both \eqref{101} and \eqref{102} are obvious, since
$$ \di(\l,e)\asymp \di(\l,\fe)\asymp |\l|\asymp \frac1{|w|}\,,
\quad \di(w,\bt_w)\asymp \di(w,E)\asymp 1 $$
by \eqref{122}. So we assume $|w|\ge r$.

Put
$$ B_j:=B_w(w_j,r)\cap\cf^{int}, \quad  B:=\bigcup B_j, $$
with small enough $r=r(\fe)$, so $B_j$ are disjoint.
The argument is based on the properties of the covering map (cf., e.g., \cite[Theorem 9.6.4]{si3}):
\begin{enumerate}
\item $\l$ can be extended analytically to a certain domain, which contains $\overline{\cf^{int}}$;
\item $\l$ is one-one in $\cf^{int}$, and $\l'(w)=0$ if and only if $w=w_j$;
\item for $w\in B_j$, we have
\begin{equation}\label{1001}
\l(w)=\l(w_j)+ C_j(w-w_j)^2+O((w-w_j)^3),
\end{equation}
and $C_j\not=0$.
\end{enumerate}
\begin{figure}[t]
\includegraphics[width=14cm]{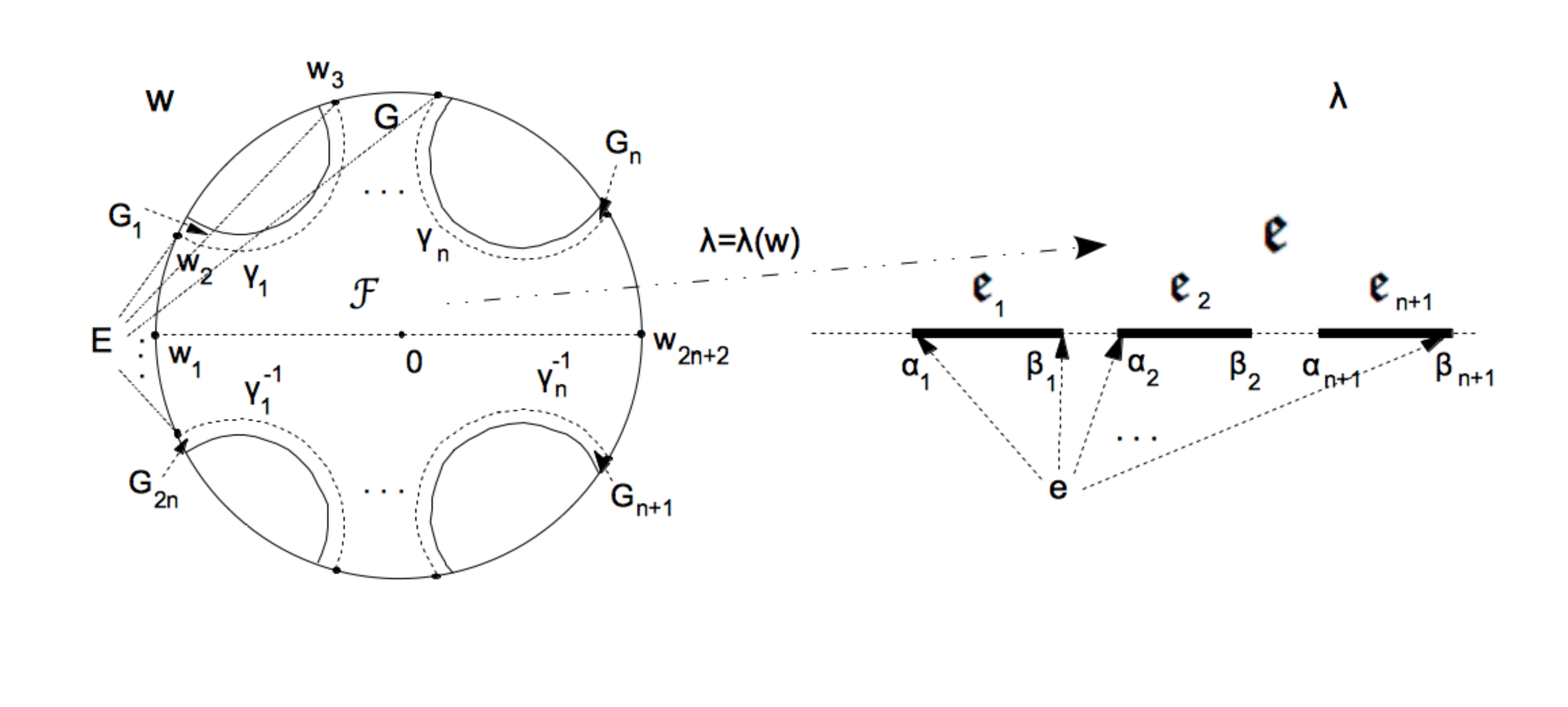}
\caption{Uniformization of the domain $\oo$ and the map $\l$.}\label{f2}
\end{figure}
By \eqref{1001}, we have for $w\in B_j$
$$ \di(\l,e)=|\l(w)-\l(w_j)|\asymp |w-w_j|^2=\di(w,E)^2\asymp
\frac{\di(w,E)^2}{|w|}\,, $$ so \eqref{101} is true on $B$.
For $w\in\overline{\cf^{int}}\backslash(B\cup B(0,r))$
$$ \di(\l,e)\asymp \di(w,E)\asymp |w|\asymp 1, $$
and the proof of \eqref{101} is complete.

To prove \eqref{102} for $|w|\ge r$ we begin with its simple half
\begin{equation}\label{1003}
\di(\l,\fe)\le C\di(w,E)\,\di(w,\bt_w).
\end{equation}
For $w\in B_j$ take $\zeta \in\bt_{\cf}=\bt\cap\overline{\cf}$ so
that $|w-\zeta|=\di(w,\bt_{\cf})$. By \eqref{1001}
\begin{equation*}
\begin{split}
|\l(\zeta)-\l(w)| &\le \max_{z\in[w,\zeta]}|\l'(z)|\,|\zeta-w| \le
C|w-w_j|\,|\zeta-w| \\ &=C\di(w,E)\,\di(w,\bt_{\cf}).
\end{split}
\end{equation*}
Since $|\l(\zeta)-\l(w)|\ge \di(\l,\fe)$ and $\di (w,\bt_w)\asymp
\di(w,\bt_{\cf})$, \eqref{1003} holds for $w\in B_j$. The similar argument applies in the case \newline $w~\in~\overline{\cf^{int}}~\backslash(B~\cup B(0,r))$, where $|\l'|\asymp 1$, so \eqref{1003} is proved.

Suppose next, that $\di(\l,\fe)\ge C\di(\l,e)$. Then by \eqref{101} for $|w|\ge r$
$$ \di(\l,\fe)\ge C\di^2(w,E)\ge C\di(w,E)\di(w,\bt_w), $$
which is opposite to \eqref{1003}, so \eqref{102} is true. Hence it remains to consider the case
\begin{equation}\label{1004}
\di(\l,\fe)\le \d\,\di(\l,e), \end{equation}
$\d$ is small enough.

We apply a version of \cite[Corollary 1.4]{pom}, which reads
\begin{equation}\label{1005}
 \di(g,\pt\oo_2)\asymp |g'(w)|\,\di(w,\pt\oo_1), \end{equation}
$g:\,\oo_1\to\oo_2$ is a conformal map of bounded domains $\oo_j$.
Let $\oo_2=B(0,R)\cap\bc_-$ be a large semidisk, such that $\fe\subset\pt\oo_2$, $g=\l$ restricted on the preimage of the later set (the part of $\cf^{int}$ in the upper half plane away from the origin). The part of \eqref{1004} in $\bc_-$ is a union $T=\cup T_j$ of small isosceles triangles $T_j$ with bases $\fe_j$. It is clear from the properties of the covering map that
$$ \di(\l,\pt\oo_2)=\di(\l,\fe), \qquad \l\in T, $$
$$ \di(w,\pt\oo_1)\asymp \di(w,\bt_w), \quad |\l'(w)|\asymp \di(w,E), \qquad w\in\l^{(-1)}(T), $$
so by \eqref{1005}
$$ \di(\l,\fe)\asymp \di(w,E)\cdot\di(w,\bt_w). $$
The proof is complete.
\end{proof}

\medskip\nt
{\it Proof of Theorem \ref{t2}.}  Let $\l=\l(w): \bd_w\to \oo_\l$
be the covering map with normalization \eqref{122}, $\gg$ the
corresponding Fuchsian group with generators $\{\g_j\}_{j=1}^n$, $E=\l^{-1}(e)$ the vertices of $\cf$. Put $\g_{2n+1-k}:=\g_k^{(-1)}$, $k=1,\ldots,n$.

Let $f\in\ca(\oo)$ satisfy \eqref{e3}. It is clear that
$|f(\infty)|=1$.  We put $F(w):=f(\l(w))$. Then $F\in\ca(\bd)$ and
automorphic with respect to $\gg$. By Lemma \ref{lem1}
\begin{equation}\label{e8}
\log |F(w)|\le \frac{K_1|w|^{p+q}}{\di^p(w,\bt) \di^{p+2q}(w,
E)}\,, \quad w\in\cf.
\end{equation}

The special structure of $\gg$ and $\cf$ enables one to "inflate"
the domain $\cf^{int}$ slightly to get another polygon $G$, so
that
$$ \cf\subset G\subset \cf\bigcup\lp\bigcup_{j=1}^{2n}\g_j(\cf)\rp, \quad \g_{n+k}(\cf)=\overline{\g_k(\cf)}, \quad k=1,\ldots,n.
$$
The distance between the corresponding inner sides of $G$ and
$\cf^{int}$ is strictly positive.

It is not hard to see that bound \eqref{e8} actually holds in the
bigger polygon $G$. Indeed, let $G_j\subset G\bsl\cf^{int}$ be an
``annular segment" between the corresponding inner sides of $G$
and $\cf^{int}$, so $G\bsl \cf^{int}=\cup^{2n}_{j=1} G_j$. We have
to check \eqref{e8} on each $G_j$. For $w\in G_j$ there is a
unique $z\in\cf^{int}$ so that $w=\g_j(z)$. Since
\begin{eqnarray*}
\di(w,\bt)&=&\di(\g_j(z),\g_j(\bt))\asymp d(z, \bt), \\
\di(z, E)&=&\di(\g_j^{-1}(w), E)\asymp \di(w, \g_j(E))\ge C\di(w,
E),
\end{eqnarray*}
where we used in an essential way that the number of generators is
finite, we see that for $w\in G_j$
\begin{eqnarray*}
\log|F(w)|&=&\log |F(z)|\le \frac{K_1|z|^{p+q}}{\di^p(z,\bt)\,
\di^{p+2q}(z,E)}\\
 &\le&\frac{CK_1|w|^{p+q}}{\di^p(w,\bt)\,\di^{p+2q}(w,E)},
\end{eqnarray*}
the first equality being exactly the automorphic property of  $F$.
Theorem~\ref{t4} with $s=p+q$ then yields
\begin{equation}\label{1111}
\sum_{w\in \ti G\cap Z_F}
\frac{\di^{p+1+\eps}(w,\bt_w)}{|w|^{(p+q-1+\eps)_+}}\,
\di^{(p+2q-1+\eps)_+}(w,E)\le C\cdot K_1
\end{equation}
for $0<\eps<1$, where $\ti G$ is another polygon with
$\cf^{int}\subset\ti G\subset G$. The more so, the same inequality
holds for $w\in\overline{\cf^{int}}\cap Z_F$.

It remains only to go back to $f\in\ca(\oo)$ and its zero set
$Z_f$. Note that although each point from $Z_f$ has infinitely
many preimages in $\bd$, we can restrict ourselves with those in
$\overline{\cf^{int}}$.  It follows easily from the properties of the covering map (see the proof of Lemma
\ref{lem1}) that $1+|\l|\asymp \frac 1{|w|}$. Hence, \eqref{101} yields
$$
\di(w,E) \asymp \lt(\frac{\di(\l,e)}{1+|\l|}\rt)^{1/2},
$$
and, with the help of \eqref{102}
$$
\di(w,\bt_w)\asymp \frac{\di(\l,\fe)}{(\di(\l,e)\, (1+|\l|))^{1/2}}.
$$
Substitution of the above relations in \eqref{1111} gives \eqref{e4}, and the proof of Theorem \ref{t2} is complete.
\hfill $\Box$

\section{Applications to complex perturbations of a finite-band selfadjoint operator}\label{s3}

Consider a bounded finite-band selfadjoint operator $A_0$, defined
on $H$.  Let $A=A_0+B$, $B\in \css_p$, with $p\ge 1$, $B$ is not
supposed to be selfadjoint.

The Schatten classes $\css_p$ form a nested family of operator
ideals, that is,
\begin{enumerate}
    \item if $p<q$, then $\css_p\subset\css_q$ and
    $\|\cdot\|_{\css_q}\le \|\cdot\|_{\css_p}$;
    \item if $P$ is a bounded operator, and $Q\in\css_p$, then
    $PQ,QP\in\css_p$ and $\|PQ\|_{\css_p}, \|QP\|_{\css_p}\le
    \|P\|\|Q\|_{\css_p}$.
\end{enumerate}
More information on the classes $\css_p$ can be found in monographs \cite{kr1} and \cite{si1}.

Given $p\ge1$ put $\lceil p\rceil:=\min\{j\in\bn:\ j\ge p\}$. The
following object known as a {\it regularized perturbation determinant}
$$
g_p(\l):=\det{}_{\lceil p\rceil} (A-\l)(A_0-\l)^{-1}
$$
is well defined, $g_p\in\ca(\oo)$,
$\oo=\overline\bc\backslash\s(A_0)$. The basic property of $g_p$
relates its zero set and the discrete spectrum of $A$:

$\l\in Z_{g_p}$ with order $k$ if and only if $\l\in\s_d(A)$ with
algebraic multiplicity $k$.

Furthermore, for $\l\in \oo$ the bound
$$ \log|g_p(\l)|\le C_p\,\|(A_0-\l)^{-1}\|^p\,\|B\|^p_{\css_p} $$
holds, see, e.g., \cite{si1}. For the selfadjoint and finite-band
operator $A_0$ the latter turns into
$$ \log|g_p(\l)|\le C_p\,\frac{\|B\|^p_{\css_p}}{\di^p(\l,\fe)}, $$
which is exactly \eqref{e1}.

Theorem \ref{t3} thus follows by a straightforward application of
Corollary~\ref{t1}.

\end{document}